%% file: main.tex
\documentclass[10pt,twocolumn,twoside,letterpaper]{IEEEtran}

\usepackage{epsfig}
\usepackage{epic, eepic}
\usepackage{pstricks}
\usepackage{graphics}
\usepackage{graphicx}
\usepackage[scanall]{psfrag}
\usepackage{subfigure}
\usepackage{here}
\usepackage{enumerate}
\usepackage{latexsym}
\usepackage{amsmath}
\usepackage{amssymb}
\usepackage{amsfonts}
\usepackage{mathrsfs}
\usepackage{euscript}
\usepackage{pifont}
\usepackage[LY1]{fontenc}                    
\usepackage{cite}

\newcommand{\ABCD}[4]{\ensuremath{\left[
                        \begin{array}{c|c}
                          #1 & #2 \\ \hline
                          #3 & #4
                        \end{array}\right]}}
\newcommand{\ABCDs}[4]{\ensuremath{\left[
                        \begin{array}{c|c}
                          #1 & #2 \\[.2ex] \hline \\[-2.25ex]
                          #3 & #4
                        \end{array}\right]}}

\newcommand{\Real}{\ensuremath{{\mathbb{R}}}}

\newcommand{\Complex}{\ensuremath{{\mathbb{C}}}}

\newcommand{\curly}[1]{\ensuremath{{\mathscr{#1}}}}

\newcommand{\Hinf}{\ensuremath{{\mathscr{H}}_{\infty}}}

\newcommand{\RHinf}{\ensuremath{{\mathscr{R}}\!{\mathscr{H}}_{\infty}}}

\newcommand{\diag}{\ensuremath{\operatornamewithlimits{diag}}}

\newcommand{\Meig}{\ensuremath{\overline{\lambda}}}

\newcommand{\adj}[1]{\ensuremath{#1^{\thicksim}}}
\newcommand{\jw}{\ensuremath{j\omega}}

\newcommand{\onehalf}{\ensuremath{\frac{1}{2}}}

\newcommand{\matrixB}[1]{\begin{bmatrix}#1\end{bmatrix}}
\newcommand{\matrixP}[1]{\begin{pmatrix}#1\end{pmatrix}}

\newcommand{\nn}{\ensuremath{\notag \\}}

\newcommand{\ds}[1]{\\[#1\baselineskip]}
\newcommand{\vsp}[1]{\vspace{#1\baselineskip}}

\newcommand{\eg}{\text{e.g.}}
\newcommand{\ie}{\text{i.e.}}

\newtheorem{theorem}{Theorem}
\newtheorem{corollary}[theorem]{Corollary}
\newtheorem{lemma}[theorem]{Lemma}

\renewcommand{\iff}{\ensuremath{\Leftrightarrow}}
\renewcommand{\implies}{\ensuremath{\Rightarrow}}
\renewcommand{\impliedby}{\ensuremath{\Leftarrow}}

\begin{document}

\title{Stability robustness of a feedback interconnection of systems with negative imaginary frequency response
       \thanks{
               Corresponding author is Alexander Lanzon.
               Tel:~+44-161-306-8722, Fax:~+44-161-306-8729,
               Email:~\texttt{a.lanzon}@\texttt{ieee.org}. \vsp{.3}
               }}

\author{Alexander Lanzon\thanks{\textsl{Alexander Lanzon} is at the
                          Control Systems Centre,
                          School of Electrical and Electronic Engineering,
                          University of Manchester, 
                          Manchester M60 1QD, 
                          UK.\vsp{.3}} \hspace{3em}
        Ian R.~Petersen\thanks{\textsl{Ian R.~Petersen} is with the 
                          School of Information Technology and Electrical Engineering, 
                          University of New South Wales at the Australian Defence Force Academy, 
                          Canberra, ACT 2600 Australia.  
                          Email:~\texttt{i.r.petersen}@\texttt{gmail.com}} \vspace{-\baselineskip}}

\markboth{Camera ready manuscript}
         {Lanzon and Petersen - Stability robustness of a feedback interconnection of \dots}

\maketitle

\begin{abstract}
  A necessary and sufficient condition, expressed simply as the DC loop gain (\ie~the loop gain at
  zero frequency) being less than unity, is given in this paper to guarantee the internal stability
  of a feedback interconnection of Linear Time-Invariant (LTI) Multiple-Input Multiple-Output (MIMO)
  systems with negative imaginary frequency response. Systems with negative imaginary frequency
  response arise for example when considering transfer functions from force actuators to co-located
  position sensors, and are commonly important in for example lightly damped structures. The key
  result presented here has similar application to the small-gain theorem, which refers to the
  stability of feedback interconnections of contractive gain systems, and the passivity theorem (or
  more precisely the positive real theorem in the LTI case), which refers to the stability of
  feedback interconnections of positive real systems.  A complete state-space characterisation of
  systems with negative imaginary frequency response is also given in this paper and also an example
  that demonstrates the application of the key result is provided.
\end{abstract}


\subsection*{\centering Notation}

Let $\RHinf^{n\times n}$ denote the set of real-rational stable transfer function matrices of
dimension $(n\times n)$. Let $\Real$ and $\Complex$ denote fields of real and complex numbers
respectively, and $\Real^{n\times n}$ and $\Complex^{n\times n}$ denote real and complex matrices
respectively of dimension $(n\times n)$. Let $\lambda_i(A)$ denote the $i$-th eigenvalue of a square
complex matrix $A$ and $\Meig(A)$ denote the maximum eigenvalue for a square complex matrix $A$ that
has only real eigenvalues. Let $\Re(a)$ and $\Im(a)$ denote the real and imaginary parts
respectively of $a\in\Complex$. Let $A^T$ and $A^*$ denote the transpose and the complex conjugate
transpose of a complex matrix $A$ and $\adj{M}(s)$ denote the adjoint of transfer function matrix
$M(s)$ given by $M(-s)^T$. Finally, let $\diag(a,b)$ be shorthand for $\matrixB{a & 0 \\ 0 & b}$ and
$A^{-*}$ be shorthand for $(A^{-1})^*$.

\section{\uppercase{Introduction}}
\label{introduction}

Consider a positive feedback interconnection of two LTI MIMO systems, $M(s)$ and $N(s)$, as
shown in Figure~\ref{fig:interconnection}, denoted by $[M(s),N(s)]$.
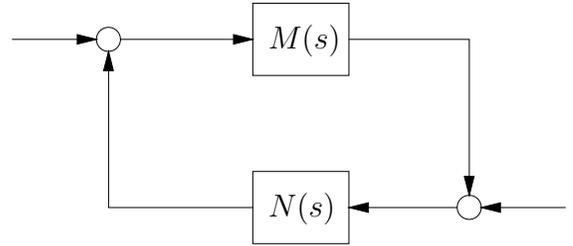
\begin{figure}[!htb]
\centering
\scalebox{1}{\input{interconnection.eepic}}
\caption{\small\label{fig:interconnection} Positive feedback interconnection}
\end{figure}
The Nyquist stability theorem (see for example~\cite{Vinnicombe:01}) gives necessary and sufficient
conditions under which this interconnection is internally stable, using much information of $M(s)$
and $N(s)$. However, when $M(s)$ and $N(s)$ satisfy certain known properties (\eg~they are both
bounded-real with product of gains less than unity, or they are both positive-real, etc), it is also
possible to derive powerful theorems (such as the small-gain
theorem~\cite{Zhou-Doyle-Glover:96,Green-Limebeer:95}, or the passivity
theorem~\cite{Anderson-Vongpanitlerd:73,Khalil:96}, etc) that use only limited information on $M(s)$
and $N(s)$ to establish the internal stability of this feedback interconnection. This is powerful
and interesting because it provides a mechanism to derive robust stability results when systems are
perturbed by uncertain dynamics that are quantified only in terms of restricted information
(\eg~stable and contractive gain for the small gain theorem, or stable and positive real for the
passivity theorem, etc).

In this paper, we derive a new result of a similar flavor. We assume that both $M(s)$ and $N(s)$
are LTI MIMO stable systems with ``negative imaginary frequency response''\footnote{Broadly, we say
  that (see Section~\ref{sec:technical_results} for precise set definitions) a system $R(s)$ has
  ``negative imaginary frequency response'' when $j[R(\jw)-R(\jw)^*]\geq 0$ (or $>0$) for all
  $\omega\in(0,\infty)$. This is because for Single-Input Single-Output (SISO) systems
  $-\Im(R(\jw))=j[R(\jw)-R(\jw)^*]$. Note that at $\omega=0$ or $\omega=\infty$, $\Im(R(\jw))=0$ as
  $R(s)$ is real-rational.}  and use this information to derive a necessary and sufficient internal
stability condition using only limited information on $M(s)$ and $N(s)$. 



We now show why systems with negative imaginary frequency response are important in
engineering applications. We will do this via a simple example. Consider a lightly damped structure
with co-located position sensors and force actuators. Lightly damped structures with co-located
position sensors and force actuators can typically be modeled by a (possibly infinite) sum of
second order transfer functions as follows:
\[  P_{\Delta}(s) := \sum_{i=1}^H \frac{k_{i}\omega_{n,i}^2}{s^2  +2\zeta_i\omega_{n,i} s  + \omega^2_{n,i}}.  \] 
For the purpose of control systems design, however, one typically tends to include only a small
finite number of modes $(h\ll H)$ in the modeling of such systems, thereby giving rise to spillover
unmodeled dynamics (\ie~unmodeled dynamics due to the lightly-damped modes not included in the
plant model).  That is, let $P(s)$ be the truncated plant model used for control systems design and
$\Delta(s)$ be the spillover dynamics, both given by:
\begin{align*}  
 P(s) &:= \sum_{i=1}^h \frac{k_{i}\omega_{n,i}^2}{s^2  +2\zeta_i\omega_{n,i} s  + \omega^2_{n,i}}
    \\ \text{ and }
\Delta(s) &:=\!\!\sum_{i=(h+1)}^H \frac{k_{i}\omega_{n,i}^2}{s^2  +2\zeta_i\omega_{n,i} s  + \omega^2_{n,i}}.
\end{align*}
It is typically an important, though difficult, design specification to ensure that
the closed-loop system retains stability in the presence of such spillover dynamics $\Delta(s)$.
Since the relative degree of such spillover dynamics is more than unity, standard positive real
analysis~\cite{Anderson-Vongpanitlerd:73,Khalil:96} will not be very helpful in establishing robust
stability and since these systems tend to be highly resonant, application of the small gain
theorem~\cite{Zhou-Doyle-Glover:96,Green-Limebeer:95} would typically be very 
conservative. However, it is readily noticed that such spillover dynamics $\Delta(s)$
are stable and satisfy a negative imaginary frequency response property. The DC gain of the
spillover dynamics is simply $\sum_{i=(h+1)}^H k_i$. Also, satisfaction of stability and the
negative imaginary frequency response property is invariant to values of $\zeta_i>0$ and $\omega_{n,i}>0$ for all
$i\in[h+1,H]$. Consequently, provided a controller $C(s)$ is designed so as to make the closed-loop
transfer function $\frac{C}{1+PC}$ from plant output disturbances to plant input satisfy a negative
imaginary frequency response property with DC gain strictly less than $1/(\sum_{i=(h+1)}^H k_i)$,
then robust stability to all spillover dynamics $\Delta(s)$, and more, will hold for any value of
$\zeta_i>0$ and $\omega_{n,i}>0$ for all $i\in[h+1,H]$, and hence $C(s)$ will also robustly
stabilize $P_{\Delta}(s)$.

Note that a similar condition  \emph{specifically for a subclass of SISO systems} has existed in the
 Positive Position Feedback control literature~\cite{GC85,FC90} for some time. It is also not
difficult to see how such a condition arises in SISO systems via a Nyquist diagram sketch. In fact,
most controller synthesis and analysis in Positive Position Feedback
control  is based on  
graphical techniques using Nyquist plots, or non-convex parameter optimization~\cite{MVB06}. In
this paper, we do three things: (a) we formalize the robustness qualities of feedback
interconnections of systems with negative imaginary frequency response via a mathematical theorem
and corollary (as opposed to graphical Nyquist sketches); (b) we extend the ideas to MIMO LTI
systems and show that even in that case, a necessary and sufficient condition for the internal stability of
such systems is that the DC loop-gain (measured in a precise sense) is less than unity; and (c) we
give a complete state-space characterization of MIMO LTI systems with negative imaginary frequency
response which may, in future work, underpin controller synthesis.

It is worth also pointing out that while a transfer function from force actuators to co-located
position sensors has typically negative imaginary frequency response, the corresponding transfer
function from force actuators to co-located velocity sensors has typically a positive real response.
Consequently, it is legitimate to wonder whether simply replacing $M(s)$ with $sM(s)$ and $N(s)$
with $-\frac{1}{s}N(s)$ in Figure~\ref{fig:interconnection} to obtain a negative feedback
interconnection and using standard positive real analysis~\cite{Anderson-Vongpanitlerd:73,Khalil:96}
would do the trick? The answer to this question is ``no, it does not'' as $\frac{1}{s}N(s)$ is not
stable, $sM(s)$ is not always guaranteed to be proper, and most importantly positive real analysis
yields an unconditional stability result whereas the interconnection of two systems with negative
imaginary frequency response will always be conditionally stable (see
Theorem~\ref{thm:main_result}). Of course, there are some connections between positive real systems
and systems with negative imaginary frequency response which in fact will be exploited in
Lemma~\ref{thm:Lyap_SS_characterisation_of_C} where we give a complete state-space characterization
of systems with negative imaginary frequency response, but the differences are also important and
should not be discounted (\eg~arising from frequencies $\omega=0$ and $\omega=\infty$).

\input{technical_results}

\input{main_result}

\input{illustrative_example}

\section{\uppercase{Conclusions}}
\label{sec:conclusions}

SISO LTI systems with negative imaginary frequency response have been studied in the context of
positive position feedback control of lightly damped structures and the analysis/synthesis
methods there depended on graphically Nyquist plots. This paper generalises the key
stability result to MIMO LTI systems with negative imaginary frequency response showing that even in
this case, a necessary and sufficient condition for the internal stability of such systems is that
the DC loop gain (measured in a particular precise sense) is less than unity. We also gave in this
paper a complete state-space characterisation of MIMO LTI systems with negative imaginary frequency
response. This could possibly be used in future work to assist with synthesising systems with
negative imaginary frequency response.

The next  steps to  extend  applicability of this research are: (a) devise a
controller synthesis procedure that generates systems that belong to either class $\curly{C}$ or
$\curly{C}_s$; and (b) generalise the analysis result given in this article to allow one (or
possibly both) systems to be nonlinear and/or time-varying. Focusing on the latter, we point out
that \cite{Angeli:06} has derived a theory for SISO nonlinear systems with counter-clockwise
input-output dynamics that is closely related to this work. It is consequently interesting to see
whether the ideas in \cite{Angeli:06} generalise to MIMO systems or not, or whether use of
dissipativity theory may lead to the required MIMO nonlinear generalisations.


\end{document}

%% file: interconnection.eepic
\setlength{\unitlength}{0.00083333in}
\begingroup\makeatletter\ifx\SetFigFont\undefined%
\gdef\SetFigFont#1#2#3#4#5{%
  \reset@font\fontsize{#1}{#2pt}%
  \fontfamily{#3}\fontseries{#4}\fontshape{#5}%
  \selectfont}%
\fi\endgroup%
{\renewcommand{\dashlinestretch}{30}
\begin{picture}(3474,1539)(0,-10)
\put(612,1287){\ellipse{150}{150}}
\put(2862,237){\ellipse{150}{150}}
\path(1512,1512)(2112,1512)(2112,1062)
	(1512,1062)(1512,1512)
\path(1512,462)(2112,462)(2112,12)
	(1512,12)(1512,462)
\path(687,1287)(1512,1287)
\blacken\path(1392.000,1257.000)(1512.000,1287.000)(1392.000,1317.000)(1392.000,1257.000)
\path(2112,1287)(2862,1287)(2862,312)
\blacken\path(2832.000,432.000)(2862.000,312.000)(2892.000,432.000)(2832.000,432.000)
\path(2787,237)(2112,237)
\blacken\path(2232.000,267.000)(2112.000,237.000)(2232.000,207.000)(2232.000,267.000)
\path(1512,237)(612,237)(612,1212)
\blacken\path(642.000,1092.000)(612.000,1212.000)(582.000,1092.000)(642.000,1092.000)
\path(12,1287)(537,1287)
\blacken\path(417.000,1257.000)(537.000,1287.000)(417.000,1317.000)(417.000,1257.000)
\path(3462,237)(2937,237)
\blacken\path(3057.000,267.000)(2937.000,237.000)(3057.000,207.000)(3057.000,267.000)
\put(1610,130){\makebox(0,0)[lb]{{\SetFigFont{12}{14.4}{\rmdefault}{\mddefault}{\updefault}$N(s)$}}}
\put(1610,1180){\makebox(0,0)[lb]{{\SetFigFont{12}{14.4}{\rmdefault}{\mddefault}{\updefault}$M(s)$}}}
\end{picture}
}

%% file: technical_results.tex
\section{\uppercase{Some technical results}}
\label{sec:technical_results}

In this section, we generate the technical machinery that will enable us to concisely prove the main result in
the next section. First, for the sake of brevity, let us define the following two sets of ``stable systems with negative
imaginary frequency response'' as follows:
  \begin{align}
    \curly{C}&:=\big\{R(s)\in\RHinf^{n\times n}: \nn &\qquad\quad\;
       j[R(\jw)-R(\jw)^*] \geq 0 \;\forall\omega\in(0,\infty)\big\},  \label{eqn:nonstrict_class} \\
    \curly{C}_s&:=\big\{R(s)\in\RHinf^{n\times n}:  \nn &\qquad\quad\;
       j[R(\jw)-R(\jw)^*] > 0 \;\forall\omega\in(0,\infty)\big\}\subset\curly{C}. \label{eqn:strict_class}
  \end{align}

The first lemma gives a complete state-space characterisation of elements in $\curly{C}$. It hence also
provides a test to easily check whether a transfer function matrix belongs to set
$\curly{C}$ or not. Testing whether a transfer function belongs to set $\curly{C}_s$ or not requires an
additional check on transmission zeros of $R(s)-\adj{R}(s)$ in the open frequency region $(0,\infty)$.

\begin{lemma} \label{thm:Lyap_SS_characterisation_of_C}
  Let $\ABCD{A}{B}{C}{D}$ be a minimal state-space realisation of a transfer  matrix
  $R(s)$. Then $R(s)\in\curly{C}$ if and only if $A$ is Hurwitz, $D=D^*$ and 
  there exists a real matrix $Y>0$ such that
  \[ AY+YA^*\leq 0 \quad\text{ and }\quad B=-AYC^*. \]    
\end{lemma}
\begin{proof}
The two statements are connected via a sequence of equivalent reformulations:
  \begin{enumerate}[(a)]
    \item $R(s)=\ABCD{A}{B}{C}{D}\in\curly{C}$.
    \item $\hat{R}(s):=(R(s)-D)=\ABCD{A}{B}{C}{0}\in\curly{C}$ and $D=D^*$.

      This equivalence follows on nothing that $R(s)\in\curly{C}$ implies
      $j[R(\infty)-R(\infty)^T] \geq 0$ via continuity and a limiting argument, which in turn 
      implies $-j[R(\infty)-R(\infty)^T]=(j[R(\infty)-R(\infty)^T])^T \geq 0$. Then, these two
      inequalities together imply $R(\infty)=R(\infty)^T$.
    \item $F(s):=s\hat{R}(s)=\ABCD{A}{B}{CA}{CB}\in\RHinf^{n\times n}$, $F(\jw)+F(\jw)^*\geq
      0\;\forall\omega\in\Real$, $A$ is Hurwitz and $D=D^*$.
    \item $A$ is Hurwitz, $D=D^*$ and $\exists X> 0, L, W$ such that
          \begin{align*}
            &XA + A^*X = - L^*L, \\
            &B^*X + W^*L = CA, \\
            &CB + (CB)^* = W^*W.
          \end{align*}
          This equivalence is via the Positive Real Lemma (see for
          example~\cite{Anderson-Vongpanitlerd:73} together with the
          fact that $(CA,A)$ is observable or~\cite[Thms~13.25,13.26]{Zhou-Doyle-Glover:96} together with the 
          fact that $(L^*L,A)$ is observable).
    \item $A$ is Hurwitz, $D=D^*$ and $\exists X>0, L, W$ such that
          \begin{align*}
            &XA + A^*X = - L^*L, \\
            &B=X^{-1}(A^*C^*-L^*W), \\
            &CX^{-1}A^*C^* + CAX^{-1}C^* \\
            &\hspace{3em}
            = W^*W + CX^{-1}L^*W + W^*LX^{-1}C^*.
          \end{align*}
    \item $A$ is Hurwitz, $D=D^*$ and $\exists X>0, L, W$ such that 
          \begin{align*}
            &XA + A^*X = - L^*L, \\
            &B=X^{-1}(A^*C^*-L^*W), \\
            &W=-LX^{-1}C^* \quad\text{via a completion of squares}.
          \end{align*}
    \item $A$ is Hurwitz, $D=D^*$ and $\exists X>0, L$ such that $XA + A^*X = - L^*L$ and $B=X^{-1}(A^*C^*+L^*LX^{-1}C^*)$.
    \item $A$ is Hurwitz, $D=D^*$ and $\exists X>0$ such that $XA + A^*X \leq 0$ and $B=-AX^{-1}C^*$.
  \end{enumerate}
\end{proof}

The second lemma relates the gain at zero frequency and the gain at infinite frequency for systems
with negative imaginary frequency response.

\begin{lemma} \label{thm:diff_between_freq_zero_and_infty}
  Given $R(s)\in\curly{C}$ (resp.~$\curly{C}_s$), then $R(0)-R(\infty) \geq 0$ (resp.~$>0$).
\end{lemma}
\begin{proof}
  Given a minimal realisation $R(s)=\ABCD{A}{B}{C}{D}\in\curly{C}$ and
  applying Lemma~\ref{thm:Lyap_SS_characterisation_of_C}, we get
  \begin{equation}
    R(0) - R(\infty) = -CA^{-1}B 
               = CA^{-1}AYC^* 
               = CYC^* \geq 0 \label{eqn:CXC}
  \end{equation}
  which concludes the proof for the non-strict inequality.

  Now, we focus on $R(s)\in\curly{C}_s\iff\hat{R}(s):=(R(s)-D)\in\curly{C}_s$ (since $D=D^*$ by
  Lemma~\ref{thm:Lyap_SS_characterisation_of_C}) and suppose there exists an
  $x\in\Real^{n\times n}$ such that $\hat{R}(0)x=0$. Then, it follows that $CYC^*x=0$ which
  implies that $C^*x=0$ as $Y>0$. This then also gives that $Bx=0$ via $B=-AYC^*$ which yields 
  \[ \hat{R}(\jw)x = C(\jw I -A)^{-1}Bx = 0 \quad\forall\omega\in\Real. \]
  But $j[\hat{R}(\jw)-\hat{R}(\jw)^*] > 0\;\forall\omega\in(0,\infty)$ implies that $\hat{R}(\jw)$
  is nonsingular for all $\omega\in(0,\infty)$ and hence the only possible $x\in\Real^{n\times n}$
  such that $\hat{R}(0)x=0$ is $x=0$. This shows that $\hat{R}(0)$ is also nonsingular and thus 
  $\hat{R}(0)>0$. This concludes the proof.
\end{proof}

The following lemma gathers some straightforward computations which help us understand properties of
systems with negative imaginary frequency response.

\begin{lemma} \label{thm:properties}
  Given $R(s)\in\curly{C}$, $\Delta(s)\in\curly{C}$ and $R_s(s)\in\curly{C}_s$. Then
  \[ R(s)+\Delta(s)\in\curly{C}\quad\text{ and }\quad R_s(s)+\Delta(s)\in\curly{C}_s. \]
\end{lemma}
\begin{proof}
Trivial.
\end{proof}

The final technical lemma provides a matrix result that states that unity is not in the spectrum of matrix
$AB$ when matrices $A$ and $B$ satisfy certain negative imaginary properties.

\begin{lemma} \label{thm:det_I-AB}
  Given $A\in\Complex^{n\times n}$ with $j[A-A^*]\geq 0$ and 
  $B\in\Complex^{n\times n}$ with $j[B-B^*]>0$. Then,
  \[ \det(I-AB)\neq 0. \]
\end{lemma}
\begin{proof}
  The suppositions can be rewritten as  $(jA) + (jA)^*\geq 0$ and $(jB)^{-1}+(jB)^{-*}>0$. Then
  $\det(I-AB) = \det(I+(jA)(jB))=\det((jA)+(jB)^{-1})\det(jB) \neq 0$.
\end{proof}

%% file: main_result.tex
\section{\uppercase{The Main Result}}
\label{sec:main_result}

The key result in this paper is Theorem~\ref{thm:main_result} below. It is an analysis theorem that states
that provided one system belongs to class $\curly{C}$ and the other system belongs to class
$\curly{C}_s$, then a necessary and sufficient condition\footnote{Under some assumptions on the
  gains of the systems at infinite frequency.} for internal stability of a positive feedback
interconnection of these two systems is to check that the DC loop gain (\ie~the loop gain at zero
frequency) is less than unity.

\begin{theorem} \label{thm:main_result}
  Given $M(s)\in\curly{C}$ and $N(s)\in\curly{C}_s$ that also satisfy $M(\infty)N(\infty)=0$
  and $N(\infty)\geq 0$. Then,
  \[ [M(s),N(s)] \text{ is internally stable} \;\:\iff\;\:
  \Meig(M(0)N(0))<1. \]  
\end{theorem}
\begin{proof}
Let $M(s)=\ABCD{A}{B}{C}{D}$ and $N(s)=\ABCD{\bar{A}}{\bar{B}}{\bar{C}}{\bar{D}}$ be minimal
realizations. Then, by the suppositions of this theorem and Lemma~\ref{thm:Lyap_SS_characterisation_of_C}, 
$A$ is Hurwitz, $D=D^*$, $\bar{A}$ is Hurwitz, $\bar{D}=\bar{D}^*\geq 0$, $D\bar{D}=0$ and 
there exists real matrices $Y>0$ and $\bar{Y}>0$ such that
\begin{gather}
  AY+YA^*\leq 0 \quad\text{ and }\quad B=-AYC^*, \label{eqn:for_P} \\
  \bar{A}\bar{Y}+\bar{Y}\bar{A}^*\leq 0 \quad\text{ and }\quad \bar{B}=-\bar{A}\bar{Y}\bar{C}^*. \label{eqn:for_Q}
\end{gather}

Now, define $\Phi:=\matrixB{AY & 0 \\ 0 & \bar{A}\bar{Y}}$ and 
$T:=\matrixB{Y^{-1}-C^*\bar{D}C & - C^*\bar{C}\\ -\bar{C}^*C & \bar{Y}^{-1}-\bar{C}^*D\bar{C}}$, and
note that:
\begin{align*}
  &\qquad\quad [M(s),N(s)] \text{ is internally stable} \ds{.2}
  &\iff\quad (I-M(s)N(s))^{-1}= \\ &\hspace{2em}
                \ABCDs{\matrixP{A & B\bar{C} \\ 0 & \bar{A}}+\matrixP{B\bar{D}\\ \bar{B}}\matrixP{C & D\bar{C}}}%
                      {\begin{array}{c} B\bar{D}\\\bar{B}\end{array}}%
                      {C\qquad\qquad\qquad D\bar{C}}{I}\in\RHinf  \displaybreak[2] \ds{.2}
  &\iff\quad \mathcal{A}:=\matrixB{A & B\bar{C} \\ 0 &
    \bar{A}}+\matrixB{B\bar{D}\\\bar{B}}\matrixB{C & D\bar{C}} =\Phi T
             \text{ is Hurwitz} \\ &\hspace{2.75em}
             [\text{as the above realization is stabilizable and detectable}]. \displaybreak[2] \ds{.2}
  &\iff\quad T>0 \\ &\hspace{2.75em}
             \parbox{7.9cm}{$[(\implies)$ Since $\mathcal{A}$ is Hurwitz and $\Phi$ is nonsingular,
             $T$ is nonsingular. Since $\Phi+\Phi^*\leq 0$, it follows that $T\mathcal{A} +
             \mathcal{A}^* T\leq 0$. Consequently, $\mathcal{A}$ is Hurwitz implies $T\geq 0$. But
             $T$ is also nonsingular, therefore $T>0$. \ds{.2} 
             $(\impliedby)$ Since  $\Phi+\Phi^*\leq 0$, it follows that $T\mathcal{A} +
             \mathcal{A}^* T\leq 0$. Consequently, $T>0$ implies $\Re(\lambda_i(\mathcal{A}))\leq
             0\;\forall i$. But $T>0$ and $\Phi$ nonsingular also imply $\mathcal{A}$
             has no eigenvalue at the origin. We now invoke Lemma~\ref{thm:det_I-AB} and use the fact that
             $M(s)\in\curly{C}$ and $N(s)\in\curly{C}_s$ to conclude that $\det(I-M(\jw)N(\jw))\neq
             0\;\forall\omega\in(0,\infty)$, which in turn is equivalent to $\mathcal{A}$ having no eigenvalue at $\jw$
             for all $\omega\in(0,\infty)$. $]$} \displaybreak[2] \ds{.2}
  &\iff\quad \bar{Y}^{-1}-\bar{C}^*D\bar{C}>0 \quad\text{and}\quad \\ &\hspace{2.75em}
             (Y^{-1}-C^*\bar{D}C) -
  C^*\bar{C}(\bar{Y}^{-1}-\bar{C}^*D\bar{C})^{-1}\bar{C}^*C > 0 \displaybreak[2] \ds{.5} 
  &\iff\quad \Meig\left[\bar Y^{\frac{1}{2}}\bar C^*D\bar C\bar
  Y^{\frac{1}{2}}\right] < 1 \quad\text{and}\quad \\ &\hspace{2.75em} 
             Y^{-1}-C^*\bar{D}C - 
  C^*(I-\bar{C}\bar{Y}\bar{C}^*D)^{-1}\bar{C}\bar{Y}\bar{C}^*C > 0 \displaybreak[2] \ds{.5} 
  &\iff\quad  \Meig\left[D\bar C\bar Y \bar C^*\right] < 1
  \quad\text{and}\quad \\ &\hspace{2.75em} 
             Y^{-1}-C^*\bar{D}C - C^*(I-N(0)D)^{-1}(N(0)-\bar{D})C > 0 \\ &\hspace{2.75em} 
             [\text{as } N(0)-\bar{D}=\bar{C}\bar{Y}\bar{C}^* \text{ via~\eqref{eqn:CXC} and } 
             \bar{D}D=0] \displaybreak[2] \ds{.5}
  &\iff\quad \Meig \left[DN(0)\right] < 1 \quad\text{and}\quad \\ &\hspace{2.75em}
             Y^{-1}-C^*(I-N(0)D)^{-1}[\bar{D} + (N(0)-\bar{D})]C > 0 \\ &\hspace{2.75em}
             [\text{as } N(0)-\bar{D}=\bar{C}\bar{Y}\bar{C}^* \text{ via~\eqref{eqn:CXC} and }
             D\bar{D}=0] \displaybreak[2] \ds{.5}
  &\iff\quad N(0)^{-1}-D>0 \quad\text{and}\quad \\ &\hspace{2.75em}
             Y^{-1}-C^*(N(0)^{-1}-D)^{-1}C > 0 \\ &\hspace{2.75em}
             [\text{as } N(0)>\bar{D} \text{ via Lemma~\ref{thm:diff_between_freq_zero_and_infty} and } 
             \bar{D}\geq 0] \displaybreak[2] \ds{.5}
  &\iff\quad \matrixB{Y^{-1} & C^* \\ C & N(0)^{-1}-D}>0 \displaybreak[2] \ds{.5}
  &\iff\quad N(0)^{-1}-D - CYC^*>0 \displaybreak[2] \ds{.5}
  &\iff\quad N(0)^{-1}-M(0)>0 \qquad \\ &\hspace{2.75em}
             [\text{as } M(0)=D+C YC^* \text{ via~\eqref{eqn:CXC}}] \ds{.5}
  &\iff\quad \Meig(M(0)N(0))<1. \ds{-2}
\end{align*}
\end{proof}

One may wonder whether Integral Quadratic Constraint (IQC) theory~\cite{Megretski-Rantzer:97}
captures the sufficiency part  of Theorem~\ref{thm:main_result} or not. This question is subtle and its answer
is non-trivial.  However, in short the answer is ``no, it does not''. The subtlety of the question arises from the
fact that IQC theory deals with the full frequency range $\omega\in\Real$ whereas systems
$N(s)\in\curly{C}_s$ satisfy the frequency domain inequality $j[N(\jw)-N(\jw)^*]>0$ \emph{only} on an open
frequency interval $\omega\in(0,\infty)$. This strict inequality cannot be satisfied at $\omega=0$
because via Lemma~\ref{thm:diff_between_freq_zero_and_infty} we know that $N(0)=N(0)^*$. This fact
causes the main theorem (Theorem~1) in~\cite{Megretski-Rantzer:97}, which underpins IQC
theory, to be inapplicable by violation of its suppositions.

The following corollary is a weaker restatement of the main theorem, written in the same form as the
small-gain theorem or the passivity theorem.

\begin{corollary} \label{thm:main_corollary}
  \begin{enumerate}[I.]
    \item Given $\gamma>0$ and $M(s)\in\curly{C}_s$ with $M(\infty)\geq 0$. Then
      $[\Delta(s), M(s)]$ is internally stable for all $\Delta(s)\in\curly{C}$ satisfying
      $\Delta(\infty)M(\infty)=0$ and $\Meig(\Delta(0))<\gamma$ (resp.~$\leq\gamma$)
      if and only if $\Meig(M(0))\leq\frac{1}{\gamma}$ (resp.~$<\frac{1}{\gamma}$).
    \item Given $\gamma>0$ and $M(s)\in\curly{C}$. Then
      $[\Delta(s), M(s)]$ is internally stable for all $\Delta(s)\in\curly{C}_s$ satisfying
      $\Delta(\infty)\geq 0$, $\Delta(\infty)M(\infty)=0$ and $\Meig(\Delta(0))<\gamma$ (resp.~$\leq\gamma$)
      if and only if $\Meig(M(0))\leq\frac{1}{\gamma}$ (resp.~$<\frac{1}{\gamma}$).
  \end{enumerate}
\end{corollary}
\begin{proof}
  Sufficiency of the two statements follows on noting that $\Meig(M(0))\Meig(\Delta(0))<1$ implies
  $\Meig(\Delta(0)M(0))<1$. Necessity can be proved via a contra-positive argument on choosing
  $\Delta(s)=\frac{1/\Meig(M(0))}{(s+1)}I$ as the destabilizing $\Delta(s)$.
\end{proof}

%% file: illustrative_example.tex
\section{\uppercase{Illustrative Example}}
\label{sec:example}

Consider the lightly damped mechanical plant depicted in Figure~\ref{fig:physical_system},
\begin{figure}[!htb]
\centering
\scalebox{.83}{\input{physical_system.eepic}}
\caption{\small\label{fig:physical_system} Lightly damped uncertain mechanical plant}
\end{figure}
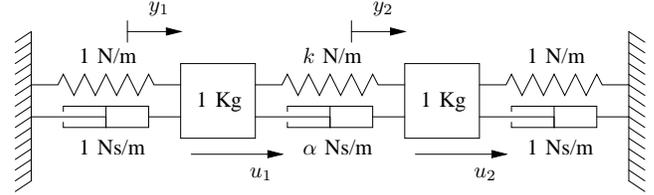
which consists of two unit masses constrained to slide rectilinearly on a frictionless
table. Each mass is attached to a fixed wall via a spring of known unit stiffness and via a
damper of known unit viscous resistance. Furthermore, the two unit masses are coupled together via a spring
of uncertain stiffness $k$~\texttt{N/m} and via a damper of uncertain viscous resistance
$\alpha$~\texttt{Ns/m}. A force is applied to each mass (denoted by $u_1$ and $u_2$ respectively)
and the displacement of each mass is measured (denoted by $y_1$ and $y_2$ respectively). 

Although this is not an extremely difficult design problem, it does illustrate a number of important
points arising from the  results in this paper, as it contains key features such as an uncertain MIMO system with
uncertainty that has negative imaginary frequency response. Similar examples have been considered
in the literature as benchmark problems by a number of authors,
including~\cite{Wie-Bernstein:92,Vinnicombe:01} to mention a few.

For shorthand, let us define some commonly appearing transfer functions and matrices. Let 
\begin{gather*} 
p(s):=\frac{1}{s^2+s+1},\quad \delta(s):=\frac{1}{s^2+(2\alpha+1)s+(2k+1)} \\
   \text{and }\; \Psi:=\matrixB{1 & 0\\ 1 & 1}. 
\end{gather*}
Then, elementary mechanical modeling reveals that the transfer function matrix for the plant
depicted in Figure~\ref{fig:physical_system} from force input vector
$u:=\matrixB{u_1 \\ u_2}$ to displacement output measurements $y:=\matrixB{y_1\\y_2}$ is given by
$y(s)=P_\Delta(s) u(s)$ where
{\footnotesize
\[ P_\Delta := p(s)\delta(s) \matrixB{(s^2+(\alpha+1)s+(k+1)) \!&\! (\alpha s +k) \\ 
                       (\alpha s +k) \!&\! (s^2+(\alpha+1)s+(k+1))}. \]}
It is clear that $P_\Delta(s)$ is uncertain because $k$ and $\alpha$ are unknown. 

For the purpose of control system design, we now choose to split the uncertain plant $P_\Delta(s)$ as
$P_\Delta(s) = P(s) + \Delta(s)$ where $P(s)$ is the nominal completely known plant model and
$\Delta(s)$ is the uncertain remainder. Via partial fraction expansion, we see that $P(s) =
\Psi\diag(\onehalf p(s),0) \Psi^*$ and $\Delta(s) = \Psi^{-1}\diag(\onehalf \delta(s),0)\Psi^{-*}$.
It is then a simple computation to check that $\Delta(s)\in\curly{C}$ for all $\alpha>0$ and $k>0$.

Now let us consider the controlled closed-loop system given in Figure~\ref{fig:closed-loop}, and let $C(s)$ be
chosen as $C(s):=\Psi^{-*}\diag(\frac{-2(s^2+s+1)}{2s^3+4s^2+4s+3},\frac{-1}{s+1})\Psi^{-1}$. Then,
define $M(s):= -C(s)(I+P(s)C(s))^{-1}$ to be the transfer function matrix mapping $w$ to $z$ so that
the closed-loop system in Figure~\ref{fig:closed-loop} can be rearranged into
Figure~\ref{fig:M_and_Delta} for robust stability analysis.
\begin{figure}[!htb]
   \centering
   \scalebox{.73}{\input{closed-loop.eepic}}
   \caption{\small\label{fig:closed-loop} Controlled closed-loop system}
   \vsp{1}
   \scalebox{.85}{\input{M_and_Delta.eepic}}
   \caption{\small\label{fig:M_and_Delta} Rearranged closed-loop}
\end{figure}

Since $P(s)\in\RHinf^{2\times 2}$, internal stability of the nominal feedback loop (\ie~pretending
$\Delta(s)=0$) is equivalent to $M(s)\in\RHinf^{2\times 2}$. Furthermore, since
$\Delta(s)\in\curly{C}$, we additionally require $M(s)\in\curly{C}_s$ to be able to apply the proposed results
Theorem~\ref{thm:main_result} and Corollary~\ref{thm:main_corollary} to conclude robust stability.

For our particular choice of $C(s)$, it is easy to see that $M(s)=\frac{1}{s+1}\Psi^{-*}\Psi^{-1}$, which
clearly belongs to $\curly{C}_s$ and furthermore satisfies $M(\infty)=0$.
Since $M(0)=\Psi^{-*}\Psi^{-1}$ (which incidentally is strictly greater than $M(\infty)$ as stated
in Lemma~\ref{thm:diff_between_freq_zero_and_infty}) and
$\Delta(0)=\Psi^{-1}\diag(\frac{1}{2(2k+1)},0)\Psi^{-*}$ (which incidentally is greater than or
equal to $\Delta(\infty)$ as stated in Lemma~\ref{thm:diff_between_freq_zero_and_infty}), it follows
that $\Meig(\Delta(0)M(0)) = \frac{5}{2(2k+1)}$, $\Meig(M(0))=\frac{3+\sqrt{5}}{2}$ and
$\Meig(\Delta(0))=\frac{1}{2k+1}$.

Consequently, Corollary~\ref{thm:main_corollary} Part I states that the feedback interconnection given in
Figure~\ref{fig:closed-loop} is robustly stable for all uncertainties $\Delta\in\curly{C}$ (not just
those of the form $\Delta(s) = \Psi^{-1}\diag(\onehalf \delta(s),0)\Psi^{-*}$) satisfying
$\Meig(\Delta(0))<\frac{2}{3+\sqrt{5}}\quad(=1/\Meig(M(0)))$. Additionally,
Theorem~\ref{thm:main_result} states that for any given $\alpha>0$, the physical system of
Figure~\ref{fig:physical_system} is robustly stabilized by the controller $C(s)$ defined above if and only
if $k>0.75$, obtained through the condition $\Meig(\Delta(0)M(0))<1$. 
The former statement is powerful because it characterizes a huge class of systems (including spillover
dynamics) for which the closed-loop system is robustly stable. The latter statement is powerful
because for a given uncertainty class, it tells us exactly in a necessary and sufficient manner
the parameter boundary of robust stability.

%% file: physical_system.eepic
\setlength{\unitlength}{0.00074366in}
\begingroup\makeatletter\ifx\SetFigFont\undefined%
\gdef\SetFigFont#1#2#3#4#5{%
  \reset@font\fontsize{#1}{#2pt}%
  \fontfamily{#3}\fontseries{#4}\fontshape{#5}%
  \selectfont}%
\fi\endgroup%
{\renewcommand{\dashlinestretch}{30}
\begin{picture}(5334,1668)(0,-10)
\path(2037,642)(2667,642)(2667,552)
	(2307,552)(2307,597)
\path(2307,687)(2307,732)(3027,732)
	(3027,552)(2667,552)(2667,732)
\path(3027,642)(3297,642)
\path(147,642)(777,642)(777,552)
	(417,552)(417,597)
\path(417,687)(417,732)(1137,732)
	(1137,552)(777,552)(777,732)
\path(1137,642)(1407,642)
\path(3927,642)(4557,642)(4557,552)
	(4197,552)(4197,597)
\path(4197,687)(4197,732)(4917,732)
	(4917,552)(4557,552)(4557,732)
\path(4917,642)(5187,642)
\path(147,1362)(147,102)
\path(147,1362)(12,1272)
\path(147,1272)(12,1182)
\path(147,1182)(12,1092)
\path(147,1092)(12,1002)
\path(147,1002)(12,912)
\path(147,912)(12,822)
\path(147,822)(12,732)
\path(147,732)(12,642)
\path(147,642)(12,552)
\path(147,552)(12,462)
\path(147,462)(12,372)
\path(147,372)(12,282)
\path(147,282)(12,192)
\path(147,192)(12,102)
\path(147,102)(12,12)
\path(5187,1362)(5187,102)
\path(5187,1362)(5322,1272)
\path(5187,1272)(5322,1182)
\path(5187,1182)(5322,1092)
\path(5187,1092)(5322,1002)
\path(5187,1002)(5322,912)
\path(5187,912)(5322,822)
\path(5187,822)(5322,732)
\path(5187,732)(5322,642)
\path(5187,642)(5322,552)
\path(5187,552)(5322,462)
\path(5187,462)(5322,372)
\path(5187,372)(5322,282)
\path(5187,282)(5322,192)
\path(5187,192)(5322,102)
\path(5187,102)(5322,12)
\path(957,1452)(957,1272)
\path(957,1362)(1407,1362)
\blacken\path(1287.000,1332.000)(1407.000,1362.000)(1287.000,1392.000)(1287.000,1332.000)
\path(2847,1452)(2847,1272)
\path(2847,1362)(3297,1362)
\blacken\path(3177.000,1332.000)(3297.000,1362.000)(3177.000,1392.000)(3177.000,1332.000)
\path(147,912)(372,912)(417,1002)
	(507,822)(597,1002)(687,822)
	(777,1002)(867,822)(957,1002)
	(1047,822)(1137,1002)(1182,912)(1407,912)
\path(2037,1092)(1407,1092)(1407,462)
	(2037,462)(2037,1092)
\path(2037,912)(2262,912)(2307,1002)
	(2397,822)(2487,1002)(2577,822)
	(2667,1002)(2757,822)(2847,1002)
	(2937,822)(3027,1002)(3072,912)(3297,912)
\path(3927,1092)(3297,1092)(3297,462)
	(3927,462)(3927,1092)
\path(3927,912)(4152,912)(4197,1002)
	(4287,822)(4377,1002)(4467,822)
	(4557,1002)(4647,822)(4737,1002)
	(4827,822)(4917,1002)(4962,912)(5187,912)
\path(1497,327)(2262,327)
\blacken\path(2142.000,297.000)(2262.000,327.000)(2142.000,357.000)(2142.000,297.000)
\path(3387,327)(4152,327)
\blacken\path(4032.000,297.000)(4152.000,327.000)(4032.000,357.000)(4032.000,297.000)
\put(550,1092){\makebox(0,0)[lb]{{\SetFigFont{10}{12.0}{\rmdefault}{\mddefault}{\updefault}1 N/m}}}
\put(550,327){\makebox(0,0)[lb]{{\SetFigFont{10}{12.0}{\rmdefault}{\mddefault}{\updefault}1 Ns/m}}}
\put(1137,1497){\makebox(0,0)[lb]{{\SetFigFont{10}{12.0}{\rmdefault}{\mddefault}{\updefault}$y_1$}}}
\put(1542,687){\makebox(0,0)[lb]{{\SetFigFont{10}{12.0}{\rmdefault}{\mddefault}{\updefault}1 Kg}}}
\put(1992,102){\makebox(0,0)[lb]{{\SetFigFont{10}{12.0}{\rmdefault}{\mddefault}{\updefault}$u_1$}}}
\put(2440,327){\makebox(0,0)[lb]{{\SetFigFont{10}{12.0}{\rmdefault}{\mddefault}{\updefault}$\alpha$ Ns/m}}}
\put(2440,1092){\makebox(0,0)[lb]{{\SetFigFont{10}{12.0}{\rmdefault}{\mddefault}{\updefault}$k$ N/m}}}
\put(3027,1497){\makebox(0,0)[lb]{{\SetFigFont{10}{12.0}{\rmdefault}{\mddefault}{\updefault}$y_2$}}}
\put(3432,687){\makebox(0,0)[lb]{{\SetFigFont{10}{12.0}{\rmdefault}{\mddefault}{\updefault}1 Kg}}}
\put(3882,102){\makebox(0,0)[lb]{{\SetFigFont{10}{12.0}{\rmdefault}{\mddefault}{\updefault}$u_2$}}}
\put(4330,327){\makebox(0,0)[lb]{{\SetFigFont{10}{12.0}{\rmdefault}{\mddefault}{\updefault}1 Ns/m}}}
\put(4330,1092){\makebox(0,0)[lb]{{\SetFigFont{10}{12.0}{\rmdefault}{\mddefault}{\updefault}1 N/m}}}
\end{picture}
}

%% file: closed-loop.eepic
\setlength{\unitlength}{0.00083333in}
\begingroup\makeatletter\ifx\SetFigFont\undefined%
\gdef\SetFigFont#1#2#3#4#5{%
  \reset@font\fontsize{#1}{#2pt}%
  \fontfamily{#3}\fontseries{#4}\fontshape{#5}%
  \selectfont}%
\fi\endgroup%
{\renewcommand{\dashlinestretch}{30}
\begin{picture}(5354,2808)(0,-10)
\put(687,1137){\ellipse{150}{150}}
\put(3762,1137){\ellipse{150}{150}}
\path(2787,2412)(3387,2412)(3387,1962)
	(2787,1962)(2787,2412)
\path(2787,1362)(3387,1362)(3387,912)
	(2787,912)(2787,1362)
\path(1287,1362)(1887,1362)(1887,912)
	(1287,912)(1287,1362)
\path(1887,1137)(2787,1137)
\blacken\path(2667.000,1107.000)(2787.000,1137.000)(2667.000,1167.000)(2667.000,1107.000)
\path(2487,1137)(2487,2187)(2787,2187)
\blacken\path(2667.000,2157.000)(2787.000,2187.000)(2667.000,2217.000)(2667.000,2157.000)
\path(3387,2187)(3762,2187)(3762,1212)
\blacken\path(3732.000,1332.000)(3762.000,1212.000)(3792.000,1332.000)(3732.000,1332.000)
\path(3387,1137)(3687,1137)
\blacken\path(3567.000,1107.000)(3687.000,1137.000)(3567.000,1167.000)(3567.000,1107.000)
\path(762,1137)(1287,1137)
\blacken\path(1167.000,1107.000)(1287.000,1137.000)(1167.000,1167.000)(1167.000,1107.000)
\path(12,1137)(612,1137)
\blacken\path(492.000,1107.000)(612.000,1137.000)(492.000,1167.000)(492.000,1107.000)
\path(4212,1137)(4212,12)(687,12)(687,1062)
\blacken\path(717.000,942.000)(687.000,1062.000)(657.000,942.000)(717.000,942.000)
\path(537,837)(612,837)
\path(3837,1137)(4587,1137)
\blacken\path(4467.000,1107.000)(4587.000,1137.000)(4467.000,1167.000)(4467.000,1107.000)
\dashline{60.000}(2337,2712)(3987,2712)(3987,612)
	(2337,612)(2337,2712)
\put(2900,1000){\makebox(0,0)[lb]{{\SetFigFont{12}{14.4}{\rmdefault}{\mddefault}{\updefault}$P(s)$}}}
\put(2900,2050){\makebox(0,0)[lb]{{\SetFigFont{12}{14.4}{\rmdefault}{\mddefault}{\updefault}$\Delta(s)$}}}
\put(1400,1000){\makebox(0,0)[lb]{{\SetFigFont{12}{14.4}{\rmdefault}{\mddefault}{\updefault}$C(s)$}}}
\put(4062,2500){\makebox(0,0)[lb]{{\SetFigFont{12}{14.4}{\rmdefault}{\mddefault}{\updefault}$P_{\Delta}(s)$}}}
\put(12,1212){\makebox(0,0)[lb]{{\SetFigFont{12}{14.4}{\rmdefault}{\mddefault}{\updefault}$r$}}}
\put(2112,1212){\makebox(0,0)[lb]{{\SetFigFont{12}{14.4}{\rmdefault}{\mddefault}{\updefault}$u$}}}
\put(2562,1662){\makebox(0,0)[lb]{{\SetFigFont{12}{14.4}{\rmdefault}{\mddefault}{\updefault}$z$}}}
\put(4212,1212){\makebox(0,0)[lb]{{\SetFigFont{12}{14.4}{\rmdefault}{\mddefault}{\updefault}$y$}}}
\put(3570,1662){\makebox(0,0)[lb]{{\SetFigFont{12}{14.4}{\rmdefault}{\mddefault}{\updefault}$w$}}}
\end{picture}
}

%% file: M_and_Delta.eepic
\setlength{\unitlength}{0.00083333in}
\begingroup\makeatletter\ifx\SetFigFont\undefined%
\gdef\SetFigFont#1#2#3#4#5{%
  \reset@font\fontsize{#1}{#2pt}%
  \fontfamily{#3}\fontseries{#4}\fontshape{#5}%
  \selectfont}%
\fi\endgroup%
{\renewcommand{\dashlinestretch}{30}
\begin{picture}(2423,1539)(0,-10)
\path(825,1512)(1425,1512)(1425,1062)
	(825,1062)(825,1512)
\path(825,462)(1425,462)(1425,12)
	(825,12)(825,462)
\path(1425,1287)(2025,1287)(2025,237)(1425,237)
\blacken\path(1545.000,267.000)(1425.000,237.000)(1545.000,207.000)(1545.000,267.000)
\path(825,237)(225,237)(225,1287)(825,1287)
\blacken\path(705.000,1257.000)(825.000,1287.000)(705.000,1317.000)(705.000,1257.000)
\put(925,130){\makebox(0,0)[lb]{{\SetFigFont{12}{14.4}{\rmdefault}{\mddefault}{\updefault}$M(s)$}}}
\put(925,1180){\makebox(0,0)[lb]{{\SetFigFont{12}{14.4}{\rmdefault}{\mddefault}{\updefault}$\Delta(s)$}}}
\put(2100,687){\makebox(0,0)[lb]{{\SetFigFont{12}{14.4}{\rmdefault}{\mddefault}{\updefault}$w$}}}
\put(0,687){\makebox(0,0)[lb]{{\SetFigFont{12}{14.4}{\rmdefault}{\mddefault}{\updefault}$z$}}}
\end{picture}
}